\documentclass[a5paper, 10pt]{article}

\usepackage{cite, amsmath, amssymb, booktabs, lastpage, graphicx}
\usepackage[hidelinks]{hyperref}

\usepackage{geometry}
\usepackage{setspace}
 \usepackage{cite, amsmath, amssymb, amsthm}
    \usepackage[mathscr]{euscript}
    \usepackage{graphicx}
    \usepackage[margin=1cm,%
                font=small,%
                format=hang,%
                labelsep=period,%
                labelfont=bf]{caption}

\usepackage{amsthm}

\theoremstyle{plain}

\newtheorem{theorem}{Theorem}
\newtheorem{conjecture}{Conjecture}
\newtheorem{proposition}[theorem]{Proposition}
\newtheorem{corollary}{Corollary}

\theoremstyle{remark}

\theoremstyle{definition}

\newtheorem{example}{Example}

\usepackage{graphicx}
\makeatletter
\renewcommand{\maketitle}{
        \begin{center}

                {\Large\bfseries \@title} \par
                \vspace{5mm}
                \baselineskip=0.2in
                {\large\bfseries \@author}\par
                \vspace{1mm}
                {\it \@address} \par
                {\small\tt \@email} \par
                \vspace{3mm}
                
        \end{center}
        \vspace{3mm}
}
\newcommand{\address}[1]{\def\@address{#1}}
\newcommand{\email}[1]{\def\@email{#1}}

\makeatother

\usepackage{fancyhdr}

\usepackage[margin=1cm,%
                        font=footnotesize,%
                        format=hang,%
                        labelsep=period,%
                        labelfont=bf]{caption}

\title{Locally Equienergetic Graphs}
\author{Cahit Dede$^{a,}$, 
        Kalpesh M. Popat$^b$}

\address{$^a$Department of Mathematics, Selçuk University, Konya, Türkiye\\
        $^b$Department of Mathematics, Saurashtra University, Rajkot, India
    }

\email{cahitdede@gmail.com, kalpeshmpopat@gmail.com}

\newcommand{\boxtensor}{{\Box\keCahitrn-9.03pt\raise1.42pt\hbox{$\times$}}}

\newcommand{\be}{\begin{eqnarray}}
\newcommand{\ee}{\end{eqnarray}}

\newcommand{\dd}{\displaystyle}

\begin{document}

\maketitle

\begin{abstract}
For a given graph \( G \), let \( G^{(j)} \) denote the graph obtained by the deletion of vertex \( v_j \) from \( G \). The difference \( \mathscr{E}(G) - \mathscr{E}(G^{(j)}) \) quantifies the change in the energy of \( G \) upon the removal of \( v_j \), termed as the local energy of \( G \) at vertex $v_j$, as defined by Espinal and Rada in 2024. The local energy of $G$ at vertex $v$ is denoted by \(\mathscr{E}_G(v)\). The local energy of the graph \( G \), therefore, is the summation of these vertex-specific local energies across all vertices in \( V(G) \), expressed by \( e(G) = \sum \mathscr{E}_G(v) \). Two graphs of the same order are defined as locally equienergetic if they have identical local energy. In this paper, we have investigated several pairs of  locally equienergetic graphs.
\end{abstract}

\onehalfspacing

\section{Introduction}  
Let \( G \) be a graph with vertex set \( \{v_1, v_2, \ldots, v_n\} \). If there is an edge between two vertices of \( G \), we say they are adjacent. The adjacency matrix \( A(G) \) of \( G \) is defined such that \( a_{ij} \) equals 1 if vertices \( v_i \) and \( v_j \) are adjacent, and 0 otherwise, for \( i, j = 1, 2, \ldots, n \). The eigenvalues of \( A(G) \) are called the eigenvalues of \( G \). The energy \( \mathscr{E}(G) \) of the graph \( G \) is defined as the sum of the absolute values of its eigenvalues. The concept of energy was introduced by Gutman \cite{gut} in 1978. A more in-depth understanding of graph energy is provided in \cite{bap,cve,li,vai}.

Let \( G^{(j)} \) be the graph obtained from \( G \) by deleting vertex \( v_j \). The quantity
\[
\mathscr{E}(G) - \mathscr{E}(G^{(j)})
\]
represents the variation in the energy of \( G \) when vertex \( v_j \) is deleted. This variation is called the local energy of \( G \) at vertex \( v_j \) and is denoted by \( \mathscr{E}_G(v_j) \). Espinal and Rada \cite{esp} have shown that \( \mathscr{E}_G(v_j) \le 2\sqrt{d_j} \), where \( d_j \) is the degree of vertex \( v_j \). Moreover, \( \mathscr{E}_G(v_j) = 2\sqrt{d_j} \) if and only if the connected component of \( G \) containing \( v_j \) is isomorphic to a star tree with \( v_j \) as its center.

The local energy of graph \( G \) is defined in \cite{esp}  as
\[
e(G) = \sum_{v \in V(G)} \mathscr{E}_G(v).
\]
Espinal and Rada \cite{esp} have obtained the local energy of several known graph families as follows:
\begin{enumerate}
    \item $e(K_n) = 2n$
    \item $e(C_n) = \begin{cases} 
        2n - 2n \cot \dfrac{\pi}{2n} + 4n \cot \dfrac{\pi}{2n}, & \text{if } n \equiv 0 \pmod{4} \\[.3cm] 
        2n - 2n \cot \dfrac{\pi}{2n} + 4n \csc \dfrac{\pi}{2n}, & \text{if } n \equiv 2 \pmod{4} \\[.3cm] 
        2n, & \text{if } n \equiv 1 \pmod{2} 
    \end{cases}$
    \item $e(K_{p,q}) = 2(p+q) \sqrt{pq} - 2p \sqrt{(p-1)q} - 2q \sqrt{(q-1)p}$
\end{enumerate}
As it is seen that $K_n$ and $C_n$ have equal local energies for odd integer $n$. This observation raises a natural question: Are there other pairs of graphs with equal local energies?

In this paper, we introduce the concept of \textit{locally equienergetic graphs}, then we give locally equienergetic graphs under disjoint union operator. Furthermore, we present an exhaustive listing of all pairs of locally equienergetic graphs of order less than 11. Our findings reveal that the only pair of locally equienergetic simple, undirected, and connected graphs is ($K_n, C_n$) for odd $n<11$.

The paper is structured as follows. In Section \ref{sec:disjoint}, we examine the behavior of local energy under the disjoint union operation. Section \ref{sec:result} provides a comprehensive list of all locally equienergetic graphs of order less than 11. Finally, we conclude our study in Section \ref{sec:conc}.

\section{Local Energy of Disjoint Union}
\label{sec:disjoint}
The local energy of the disjoint union of two graphs is equal to the sum of the local energies of the two individual graphs.
\begin{theorem} \label{thm:u} Let $G$ and $H$ be two graphs then
$e(G\cup H)= e(G)+e(H)$. 
\end{theorem}
\begin{proof}
Let \( G \) and \( H \) be two graphs with adjacency matrices \( A(G) \) and \( A(H) \), respectively. The adjacency matrix of the disjoint union of \( G \) and \( H \), denoted \( G \cup H \), is given by

\[
A(G \cup H) = \begin{pmatrix} A(G) & 0 \\ 0 & A(H) \end{pmatrix}.
\]
The local energy of the graph $G \cup H$ is given by
\[
\begin{array}{lcl}
e(G \cup H) &=& \dd\sum_{v \in V(G \cup H)} \mathscr{E}_{G \cup H}(v)\\[.6cm]
&=& \dd\sum_{v \in V(G)} \mathscr{E}_{G \cup H}(v) + \sum_{v \in V(H)} \mathscr{E}_{G \cup H}(v)\\[.6cm]
&=& \dd\sum_{v \in V(G)}{\mathscr{E}({G \cup H})-\mathscr{E}({G^v \cup H})} + \sum_{v \in V(H)} {\mathscr{E}({G \cup H})-\mathscr{E}({G \cup H^v})}\\[.6cm]
&=& \dd\sum_{v \in V(G)}{\mathscr{E}(G) +\mathscr{E}(H)-\mathscr{E}(G^v) -\mathscr{E}(H)}  \\[.6cm]
& + & \sum_{v \in V(H)}{\mathscr{E}(G) +\mathscr{E}(H)-\mathscr{E}(G) -\mathscr{E}(H^v)} \\[0.6cm]
&=& \dd\sum_{v \in V(G)}{\mathscr{E}(G) -\mathscr{E}(G^v)} + \sum_{v \in V(H)}{\mathscr{E}(H) -\mathscr{E}(H^v)}\\[.6cm]
&=& e(G) + e(H).
\end{array}
\]
This completes the proof.
\end{proof}
\section{Locally Equienergetic Graphs} \label{sec:result} 
Two graphs \( G_1 \) and \( G_2 \) of the same order are defined as equienergetic if \( \mathscr{E}(G_1) = \mathscr{E}(G_2) \). More results related to equienergetic graphs are given in \cite{vai2,bran,mili,ram}.  Extending this concept, we introduce the notion of locally equienergetic graphs. Two graphs \( G_1 \) and \( G_2 \) of the same order are said to be locally equienergetic if \( e(G_1) = e(G_2) \). By applying the preceding theorem, we now construct locally equienergetic graphs composed of unions of complete graphs \( K_n \).
\begin{theorem} \label{thm:kn}
Let \( n \in \mathbb{Z}^+ \) and suppose \( n \) can be decomposed as \( n = n_1 + n_2 + \cdots + n_k \). Then, the local energy of the complete graph \( K_n \) is   
\[
e(K_n) = e(K_{n_1} \cup K_{n_2} \cup \cdots \cup K_{n_k}).
\]
\end{theorem}

\begin{proof}
It is known that \( e(K_n) = 2n \) for the complete graph \( K_n \) with \( n \) vertices. Let \( n = n_1 + n_2 + \cdots + n_k \)  and consider the graph \( K_{n_1} \cup K_{n_2} \cup \cdots \cup K_{n_k} \), which represents the disjoint union of the complete graphs \( K_{n_1}, K_{n_2}, \ldots, K_{n_k} \). By the additivity property of local energy over disjoint unions, we have
\[
e(K_{n_1} \cup K_{n_2} \cup \cdots \cup K_{n_k}) = e(K_{n_1}) + e(K_{n_2}) + \cdots + e(K_{n_k}).
\]
Since \( e(K_{n_i}) = 2n_i \) for each \( i \), it follows that
\[
e(K_{n_1} \cup K_{n_2} \cup \cdots \cup K_{n_k}) = 2n_1 + 2n_2 + \cdots + 2n_k = 2(n_1 + n_2 + \cdots + n_k) = 2n.
\]
Therefore, we conclude that
\[
e(K_n) = e(K_{n_1} \cup K_{n_2} \cup \cdots \cup K_{n_k}),
\]
which completes the proof.
\end{proof}
To record the construction of the simplest locally equienergetic graphs, we present the following corollary.
\begin{corollary} \label{cor:kn}
For all \( k \in \mathbb{Z}^+ \) such that \( 2 \leq k \leq n - 2 \), we have
\[
e(K_n) = e(K_{n-k} \cup K_k).
\]
\end{corollary}
\begin{example} The following graphs of same order have same local energies:
\begin{enumerate}
    \item[a)] \( e(K_4) = e(K_2 \cup K_2) \)
    \item[b)] \( e(K_5) = e(K_3 \cup K_2) \)
    \item[c)] \( e(K_6) = e(K_4 \cup K_2) = e(K_3 \cup K_3) = e(K_2 \cup K_2 \cup K_2) \)
\end{enumerate}
\end{example}

Given that \( e(K_n) = 2n \) for all \( n \geq 2 \) and \( e(C_n) = 2n \) when \( n \) is odd, we conclude that \( K_n \) and \( C_n \) are locally equienergetic for odd \( n \). For convenience, we define the local energy of the trivial graph \( K_1 \) (a single isolated vertex) as zero, i.e., \( e(K_1) = 0 \). This definition implies that adding an isolated vertex to any graph \( G \) does not alter its local energy: 
\be \label{eqn:k1}
e(G \cup K_1) = e(G).
\ee

Based on an exhaustive computer search for all graphs of order \( \leq 10 \), we have identified the following sets of locally equienergetic graphs. Their proofs directly follow from Theorem \ref{thm:u}, Theorem \ref{thm:kn}, Corollary \ref{cor:kn} and Equation \eqref{eqn:k1}. Let  \( P_n \) denote the path graph of order $n$, and   \( S_n \) denote the star graph of order $n$.

\begin{proposition}
The only locally equienergetic graphs of order 4 are \( K_4 \) and \( K_2 \cup K_2 \), i.e., \( e(K_4) = e(K_2 \cup K_2) \).
\end{proposition}

\begin{proposition}
There are 2 classes of locally equienergetic graphs of order 5:
\begin{itemize}
    \item [(i)] \( e(K_5) = e(K_3 \cup K_2) = e(C_5) \)
    \item [(ii)] \( e(K_4 \cup K_1) = e(K_2 \cup K_2 \cup K_1) \)
\end{itemize}
\end{proposition}

\begin{proposition}
There are 3 classes of locally equienergetic graphs of order 6:
\begin{itemize}
    \item [(i)] \( e(K_6) = e(K_4 \cup K_2) = e(K_3 \cup K_3) = e(K_2 \cup K_2 \cup K_2) \)
    \item [(ii)] \( e(K_5 \cup K_1) = e(K_3 \cup K_2 \cup K_1) = e(C_5 \cup K_1) \)
    \item [(iii)] \( e(K_4 \cup K_1 \cup K_1) = e(K_2 \cup K_2 \cup K_1 \cup K_1) \)
\end{itemize}
\end{proposition}

\begin{proposition}
There are 5 classes of locally equienergetic graphs of order 7:
\begin{itemize}
    \item [(i)] \( e(K_7) = e(K_5 \cup K_2) = e(K_4 \cup K_3) = e(K_3 \cup K_2 \cup K_2) = e(C_7) = e(C_5 \cup K_2) \)
    \item [(ii)] \( e(K_6 \cup K_1) = e(K_4 \cup K_2 \cup K_1) = e(K_3 \cup K_3 \cup K_1) = e(K_2 \cup K_2 \cup K_2 \cup K_1) \)
    \item [(iii)] \( e(K_5 \cup K_1 \cup K_1) = e(K_3 \cup K_2 \cup K_1 \cup K_1) = e(C_5 \cup K_1 \cup K_1) \)
    \item [(iv)] \( e(K_4 \cup K_1 \cup K_1 \cup K_1) = e(K_2 \cup K_2 \cup K_1 \cup K_1 \cup K_1) \)
    \item [(v)] \( e(K_4 \cup P_3) = e(K_2 \cup K_2 \cup P_3) \)
\end{itemize}
\end{proposition}

\begin{proposition}
There are 12 classes of locally equienergetic graphs of order 8:
\begin{itemize}
    \item [(i)] \( e(K_8) = e(K_6 \cup K_2) = e(K_5 \cup K_3) = e(2K_4) = e(2K_3 \cup K_2) = e(C_5 \cup K_3)  = e(4K_2) \)
    \item [(ii)] \( e(K_7 \cup K_1) = e(K_5 \cup K_2 \cup K_1) = e(K_4 \cup K_3 \cup K_1) = e(K_3 \cup 2K_2 \cup K_1) = e(C_7 \cup K_1) = e(C_5 \cup K_2 \cup K_1) \)
    \item [(iii)] \( e(K_6 \cup 2K_1) = e(K_4 \cup K_2 \cup 2K_1) = e(K_3 \cup K_3 \cup 2K_1) = e(3K_2 \cup 2K_1) \)
    \item [(iv)] \( e(K_5 \cup 3K_1) = e(K_3 \cup K_2 \cup 3K_1) = e(C_5 \cup 3K_1) \)
    \item [(v)] \( e(K_4 \cup 4K_1) = e(2K_2 \cup 4K_1) \)
    \item [(vi)] \( e(K_4 \cup P_3 \cup K_1) = e(2K_2 \cup P_3 \cup K_1) \)
    \item [(vii)] \( e(K_4 \cup C_4) = e(2K_2 \cup C_4) \)
    \item [(viii)] \( e(K_4 \cup C'_4) = e(2K_2 \cup C'_4) \), where $C'_4$ is the graph of 4-cycle with a diagonal.
    \item [(ix)] \( e(K_4 \cup S_4) = e(2K_2 \cup S_4) \)
    \item [(x)] \( e(K_4 \cup P_4) = e(2K_2 \cup P_4) \)
    \item [(xi)] \( e(K_4 \cup P'_4) = e(2K_2 \cup P'_4) \), where $P'_4$ is the graph with 3-cycle with a tail. 
    \item [(xii)] \( e(K_5 \cup P_3) = e(K_3 \cup K_2 \cup P_3) =e(C_5 \cup P_3) \)
\end{itemize}
\end{proposition}

A similar classification of locally equienergetic graphs exists for orders 9 and 10. Summarizing our exhaustive search, we establish the following:

\begin{corollary}
There exists no simple, undirected, and connected locally equienergetic graph of order \( n \leq 10 \) except the pair \( (K_n,C_n) \) for odd \( n \).
\end{corollary}

This observation leads to the following conjecture:

\begin{conjecture}
There exists no simple, undirected, and connected locally equienergetic graph other than the pair of locally equienergetic graphs \( (K_n, C_n) \) for odd \( n \).
\end{conjecture}

\section{Concluding Remarks} \label{sec:conc}
In conclusion, this paper has explored the structure and properties of locally equienergetic graphs, providing significant insights into the concept of local energy in graph theory. The introduced concept of locally equienergetic graphs has potential implications for the analysis of graph robustness, particularly in contexts where the impact of vertex removal is relevant, such as network resilience and vulnerability assessments.

\singlespacing

\end{document}